\documentclass{article}

\usepackage{amsthm}

\newtheorem{theorem}{Theorem}[section]

\newtheorem{proposition}[theorem]{Proposition}
\newtheorem{remark}[theorem]{Remark}

\usepackage{graphicx, amsmath, color}
\usepackage{amssymb}
\usepackage{booktabs}
\usepackage{float}
\usepackage{psfrag}

\numberwithin{equation}{section}
\usepackage[left=3.5cm,top=4cm,right=3.3cm,bottom=3.3cm]{geometry}

\usepackage[dvips,
  bookmarksopen,
  colorlinks,
  linkcolor = black,
  urlcolor  = blue,
  citecolor = black,
  menucolor = blue
]{hyperref}

\DeclareMathOperator{\hankel}{\mathbf H}
\DeclareMathOperator{\fourier}{\mathbf  F}
\DeclareMathOperator{\sine}{\mathbf  S}
\DeclareMathOperator{\cosine}{\mathbf C}
\DeclareMathOperator{\sign}{sign}
\DeclareMathOperator{\supp}{supp}

\newcommand{\Rrc}{R}
\newcommand{\Fe}{\operatorname{F}}

\newcommand{\R}{\mathbb{R}}
\newcommand{\N}{\mathbb{N}}
\newcommand{\Z}{\mathbb{Z}}

\newcommand{\f}{\mathbf}
\newcommand\set[1]{\left\{#1\right\}}
\newcommand\abs[1]{\left|#1\right|}
\newcommand\norm[1]{\left\|#1\right\|}

\newcommand{\Ge}{{\tt G}}
\newcommand{\Seo}{{\tt S}}
\newcommand{\Feo}{{\tt F}}
\newcommand{\mo}{{\tt m}}
\newcommand{\No}{{\tt N}}
\newcommand{\no}{{\tt n}}

\newcommand{\lo}{{\tt l}}

\newcommand{\req}[1]{(\ref{eq:#1})}

\title{Exact Series Reconstruction in Photoacoustic Tomography with Circular Integrating Detectors}

\begin{document}

\author{Gerhard Zangerl\thanks{Department of Mathematics University of Innsbruck,
Technikerstra\ss{}e~21a, 6020 Innsbruck (\href{mailto:Gerhard.Zangerl@uibk.ac.at}{\tt Gerhard.Zangerl@uibk.ac.at})}
\and  Otmar Scherzer\thanks{Department of Mathematics University of Innsbruck, Technikerstra\ss{}e~21a,
6020 Innsbruck, and Radon Institute of Computational and Applied Mathematics,
Altenberger Stras\ss{}e~69, 4040 Linz, Austria
(\href{mailto:Otmar.Scherzer@uibk.ac.at}{\tt Otmar.Scherzer@uibk.ac.at})}
\and Markus Haltmeier\thanks{Department of Mathematics, University of Innsbruck, Technikerstra\ss{}e~21a,
6020 Innsbruck (\href{mailto:Markus.Haltmeier@uibk.ac.at}{\tt Markus.Haltmeier@uibk.ac.at})}
}

\pagestyle{myheadings} \markboth{Series Reconstruction in PAT with
Circular Integrating Detectors}{G. Zangerl, O. Scherzer, and M.
Haltmeier}

\maketitle


\begin{abstract}
A method for photoacoustic tomography is presented that uses
circular integrals of the acoustic wave for the reconstruction of a
three-dimensional image. Image reconstruction is a two-step process:
In the first step data from a stack of circular integrating
detectors are used to reconstruct the circular projection of the
source distribution. In the second step the inverse circular Radon
transform is applied. In this article we establish inversion
formulas for the first step, which involves an inverse problem for
the axially symmetric wave equation. Numerical results are presented
that show the validity and robustness  of the resulting algorithm.



\end{abstract}

\section{Introduction}\label{sec:1}

The principle of Photoacoustic tomography (PAT), also called
Thermoacoustic tomography, is based on the excitation of high
bandwidth acoustic  waves with pulsed non-ionizing electromagnetic
energy inside tissue \cite{FinRak07, KucKun08, PatSch07, SchGraGroHalLen08, XuWan06}. PAT
presents a hybrid imaging technique that combines the advantages of
optical (high contrast) and ultrasound imaging (high resolution).
It has proven  great potential for important medical applications including cancer diagnostics
\cite{KruMilReyKisReiKru00,ManKhaHesSteLee05} and
imaging of vasculature \cite{EseLarLarDeyMotPro02, KolHonSteMul03}.

The common approach uses small conventional piezoelectric
transducers that approximate point detectors to measure the acoustic
pressure  \cite{XuWan06}. Reconstruction algorithms which are based
on the point detector assumption yield images with a spatial
resolution that is essentially limited by the size of the used
piezoelectric transducers \cite{XuWan03}. The size of the detector
can in principle be reduced, but only at the cost of also reducing
the signal-to-noise ratio. In order to overcome this limitation
large size planar or  linear  integrating detectors have been
proposed in \cite{BurHofPalHalSch05, HalSchBurPal04}. Line shaped detectors
integrate the acoustic pressure over its length and can be
implemented by a Mach-Zehnder \cite{PalNusHalBur07} or a Fabry-Perot
interferometer \cite{GruHalPalBur07}, which naturally integrate the
acoustic pressure over the length of a laser beam.

A drawback of linear integrating detectors is that because of
attenuation parts of the detector which are distant from the object
may be less influenced by the pressure wave. Moreover, linear
integrating detectors  do not provide a compact
experimental buildup.
In \cite{ZanSchHal09} so called \emph{circular integrating
detectors} where introduced, which integrate the acoustic pressure
over circles. A circular integrating detector can, similar to a
linear integrating detector, be implemented by an interferometer
where the laser beam is guided along a circle in an optical fiber.
It is free of any aperture effect and can provide a uniformly and
high resolution throughout the image area. Since it is possible to
fabricate optical fibers out of materials which have nearly the same
acoustical density like the fluid in which they are contained
\cite{GruHalPalBur07} no shadowing effects due to the circular
integrating detectors are expected. The use of circular integrating
detectors has been proposed independently in \cite{YanWan07}.
However, their study was limited to two spatial dimensions, where
the circular shaped detector can be used as virtual point detector.

\begin{psfrags}
\psfrag{r}{$r$} \psfrag{z}{$z$} \psfrag{r0}{$r_{\rm det}$}
\psfrag{stack}{stack}
\psfrag{sig}{$\sigma$}
\psfrag{object}{object}
\psfrag{Omega}{$B_{\Rrc}(0)\times \R$}
\begin{figure}
\begin{center}
\includegraphics[width=0.6\textwidth]{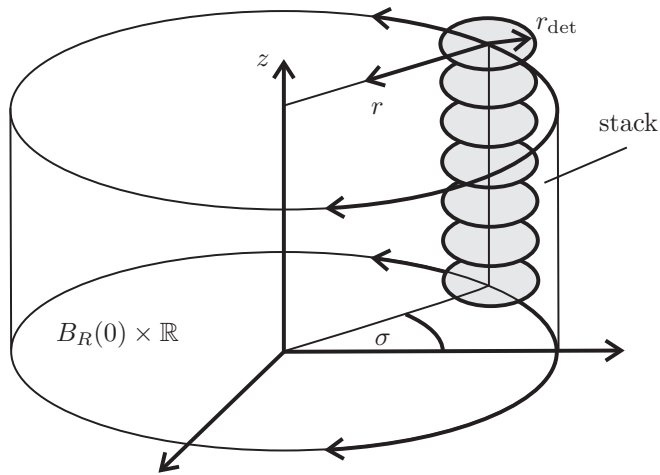}
\caption{Scanning geometry: A stack of circles centered on the boundary of
$B_{\Rrc}(0)\times \R$ is rotated around
the $\f e_3$-axis.}\label{fig:1}
\end{center}
\end{figure}
\end{psfrags}

In \cite{ZanSchHal09} we derived an inversion formula based on the expansion
of the involved functions in special basis-functions. This formula,
is  hard to implement directly due to possible division by a zero (see Subsection \ref{sec:rec}).
In this article we prove two novel exact series solutions that allow for
stable implementation. As a byproduct we obtain a  novel inversion formula for
PAT using point-like detectors on a cylindrical
recording surface (see Remark \ref{r0is0}).

The outline of this article is as follows. In Section
\ref{sec:cyl-geo} we review PAT with circular integrating detectors
and recall the main results of \cite{ZanSchHal09}. In the following
Sections  we will be concerned with derivation and implementation of
the novel stable inversion formulas.

\section{Photoacoustic Tomography with Circular Integrating Detectors}
\label{sec:cyl-geo}

In PAT an acoustic pressure wave $ p(\f x,t)$, inside an object, is
generated by a pulse of non-ionizing electromagnetic radiation. In
the case of spatially constant sound speed, the induced acoustic
pressure satisfies  the initial value problem
\cite{SchGraGroHalLen08,XuWan06}
\begin{align}\label{eq:wave3d}
\partial_t^2 p(\f x,t)
&=  \triangle p(\f x,t) \,, & &(\f x,t)  \in \R^3 \times (0,
\infty)\,,
\\ \label{eq:ini3da}
p(\f x, 0) &= f(\f x) \,, & &\f x  \in \R^3 \,,
\\ \label{eq:ini3db}
\partial_t p(\f x, 0) &= 0 \,, & &\f x  \in \R^3  \,.
\end{align}
PAT is concerned with recovering the initial pressure from
measurements of $p$ outside the support of $f$.

In \cite{ZanSchHal09} it is proposed to measure the acoustic signals
with a stack of parallel circles which is rotated around a single
axis, see Figure \ref{fig:1}. In such a situation, three-dimensional
imaging involves the inversion of the classical circular Radon
transform and the inversion of a reduced wave equation, as outlined
in the following.

Throughout this article it is assumed that $f$ is smooth and
supported in the cylinder $B_\Rrc(0) \times \R$, where $R$ is a
fixed positive number. Let $p(x,t)$ denote the unique solution of
(\ref{eq:wave3d})-(\ref{eq:ini3db}) and, for $\sigma \in S^1$,
define
\begin{align}
  P_\sigma ( z, r, t )
  &:=
  \frac{1}{2\pi}
  \int_{0}^{2\pi}
  p( \f\Phi_\sigma ( z, r, \alpha), t ) \, d\alpha \,,
  & & ( z, r, t) \in \R \times (0, \infty)^2 \,, \label{eq:Psigma} \\
  F_\sigma ( z, r )
  &:=
  \frac{1}{2\pi}
  \int_{0}^{2\pi}
  f( \f\Phi_\sigma ( z, r, \alpha) ) \, d\alpha \,,
  & & (z, r) \in \R \times (0, \infty) \,, \label{eq:Fsigma}
\end{align}
where
\begin{align*}
     \f\Phi_\sigma
     ( z, r, \alpha)
     =
    R \sigma
    +
    \left( r \cos(\alpha), r \sin(\alpha), z \right)^T\,,
    \quad ( z, r, \alpha ) \in \R \times (0,\infty)  \times  [0, 2\pi]\,.
\end{align*}
The stack of circular integrating detectors measures
\begin{equation}\label{eq:data}
    G_\sigma(z, t)
    :=
    P_\sigma( z, r_{\rm det}, t)\,,
    \qquad (\sigma, z, t) \in S^1 \times \R \times (0,\infty)  \,,
\end{equation}
with $r_{\rm det} >0$ denoting the fixed radius of the detectors.

The goal is to recover the unknown initial data $f$ from measured
data $(G_\sigma)_{\sigma \in S^1}$.

\subsection{Two Stage Reconstruction}

Reconstruction  with circular integrating detectors is based on the
following reduction to the axial symmetric wave equation:

\begin{proposition}\label{thm:decomp}
Let $f \in C_0^\infty(B_{\Rrc}(0)\times \R)$ and define $P_\sigma$ and $F_\sigma$,
$\sigma \in S^1$  by (\ref{eq:Psigma}), (\ref{eq:Fsigma}). Then
$P_\sigma$ satisfies the axial symmetric wave equation
\begin{align} \label{eq:wave2d}
    \partial_t^2 P_\sigma(z, r, t)
    &=  \left( r^{-1}\partial_r r\partial_r + \partial_z^2 \right)
    P_\sigma (z, r, t) \,, & & (z, r, t)  \in \R \times (0, \infty)^2\,,
    \\
    \label{eq:ini2da}
    P_\sigma(z, r, 0) &= F_\sigma(z, r)\,, & & (z, r)  \in \R \times (0, \infty) \,,
    \\
    \label{eq:ini2db}
    \partial_t P_\sigma(z, r, 0) &= 0 \,, & & (z, r )  \in \R \times (0, \infty) \,.
\end{align}
Moreover $P_\sigma$ remains bounded as  $r\to 0$.
\end{proposition}

\begin{proof}
Equations (\ref{eq:ini2da}), (\ref{eq:ini2db}) and the  boundedness as  $r\to 0$
immediately follow from (\ref{eq:ini3da}), (\ref{eq:ini3db}) and the definitions of
$P_\sigma$ and $F_\sigma$. In the cylindrical coordinates
$\f\Phi_\sigma$, the Laplace operator is given by the well known
expression
\begin{equation*}
\Delta
    =
    r^{-1}\partial_r r\partial_r
    + \partial_z^2
    +
    r^{-2}
    \partial_\alpha^2.
\end{equation*}
Integrating the equation $ \Delta p = \partial^2_tp$ with respect to
$\alpha$ yields (\ref{eq:wave2d}).
\end{proof}

Note that (\ref{eq:wave2d})-(\ref{eq:ini2db}) is uniquely solvable if we require that
its solution remains bounded as $r \to 0$.

\begin{remark}
Proposition \ref{thm:decomp} is the basis of the following two-stage
procedure which reconstructs the initial pressure $f$ in
(\ref{eq:wave3d})-(\ref{eq:ini3db}) from the data
$(G_\sigma)_{\sigma \in S^1}$:

\begin{enumerate} \label{rem:twostep}
\item[(i)]
For $\sigma\in S^1$ (fixed position of the stack of circles)
determine the initial pressure $F_{\sigma}$ of
(\ref{eq:wave2d})-(\ref{eq:ini2db}) using data $G_\sigma$.

Repeating this procedure for every $\sigma$, one obtains a
family of functions $F_\sigma$, $\sigma \in S^1$, corresponding
to averages over circles centered on $\partial
(B_{\Rrc}(0)\times \R)$.

\item[(ii)] \label{i2}
Next one recognizes that for fixed $z= z_0$, the function
\begin{equation*}
   ( \sigma , r )  \mapsto F_\sigma(  z_0,r )
\end{equation*}
is the circular mean transform of $f|_{\set{z= z_0}}$ with
centers on a circle.
For the circular mean transform stable analytic inversion
formulas have been discovered recently  \cite{AgrKucQui07, FinHalRak07,
HalSchBurNusPal07, Kun07, Kun07b}. Exemplarily, one of the formulas in
\cite[Theorem 1.1]{FinHalRak07} reads
\begin{equation}\label{eq:inv-finchi}
    f(\f x',z_0)
    =
    \frac{1}{2\pi}
    \int_{S^1}
        \left(
            \int_{0}^{2 R}
            (\partial_r r \partial_r  F_\sigma)( z_0,r) \log
            \left|
                r^2- | \f x' - R\sigma|^2
            \right| \, dr
        \right)
            d\sigma \,,
    \end{equation}
where $\f x'$ denotes the coordinates in the plane $\R^2 \times \{
z_0  \}$.
\end{enumerate}
\end{remark}

\smallskip
The key task for reconstruction $f$ is to derive stable and fast
algorithms to reconstruct the initial data in (\ref{eq:wave2d})-(\ref{eq:ini2db})
from measurement data $G_\sigma$.  A possible reconstruction method
is based of time reversal (back-propagation)  similar to \cite{BurMatHalPal07, ClaKli07, HriKucNgu08}.
However, the degeneration of $r^{-1}\partial_r r\partial_r$ at $r=0$ and the open detector set
may cause difficulties in such procedures. The inversion approach in this paper is based on  analytic inversion formulas
for reconstructing $F_\sigma$.

\subsection{Exact Inversion Formula}\label{sec:rec}

In the following denote by
\begin{align*}
\sine \set{ \phi } (\omega)
&:=
\int_0^\infty \phi ( t ) \sin( \omega t)   dt \,,
&& \qquad \phi \in L^1((0, \infty)) \,, \omega > 0\,,
\\
\cosine \set{ \phi } (\omega)
&:=
\int_0^\infty \phi ( t ) \cos( \omega t)   dt \,,
&&\qquad \phi \in L^1((0, \infty)) \,, \omega > 0\,,
\\
\fourier \set{\phi}(k)
&
:= \int_\R  \phi ( z ) e^{-ikz} dz\,,
&&\qquad \phi \in L^1(\R) \,, k \in \R \,,
\\
\hankel
\set{ \phi } ( v ) &:=  \int_0^\infty \phi(r) J_0(vr) \ rdr \,,
&&\qquad \phi \in L^1((0, \infty), r^{1/2} dr) \,, v > 0 \,,
\end{align*}
the sine, cosine, Fourier, and Hankel transform, respectively.
(Here $J_0$ is the zero order  Bessel function \cite{AbrSte72}.)
When the above transforms are applied to functions depending on several variables then the
transformed variable is added as subscript, e.g., $\hankel_r \set{
F_\sigma } ( z, v ) =  \int_0^\infty F_\sigma (z, r) J_0(vr) rdr$.
\medskip

\begin{proposition} \label{thm:inv}
Let $f \in C_0^\infty(B_{\Rrc}(0)\times \R)$ and define $F_\sigma(z,r) $ and
$G_\sigma(z,t)$ by (\ref{eq:Fsigma}), (\ref{eq:data}). Then the
relation
\begin{eqnarray}\label{invfo}
    \hankel_r \set{ \fourier_z \set{ F_\sigma }}(k, v)
    =
    \frac{2}{\pi} \frac{\cosine_t \set{ \fourier_z \set{ G_\sigma }} ( k, \sqrt{k^2 + v^2 } )  }{J_0(
    r_{\rm det} v)  \sqrt{k^2+ v^2}}
\end{eqnarray}
holds whenever $J_0(r_{\rm det} v) \neq 0 $.
\end{proposition}

\begin{proof}
The zero order Bessel function satisfies $r^{-1} \partial_r r \partial_r J_0(r) = - J_0(r) $ on $(0,\infty)$. Hence the chain rule implies
that $r^{-1} \partial_r r \partial_r J_0(rv) = - v^2 J_0(rv) $ for every $r, v>0$.
Separation of variables shows that the functions
\begin{equation*}
    (z,r,t) \mapsto e^{i k z} \cos\bigl( t \sqrt{k^2 + v^2}  \bigr) J_0(r v),
    \qquad (k,v) \in \R \times (0, \infty)\,,
\end{equation*}
solve (\ref{eq:wave2d}), (\ref{eq:ini2db}). Employing
the initial condition (\ref{eq:ini2da}) and the inversion formulas for the
Fourier and Hankel transforms it follows that the unique bounded solution of (\ref{eq:wave2d})-(\ref{eq:ini2db})
is given by
\begin{equation} \label{sol2d}
P_\sigma (r,z,t) := \frac{1}{2\pi} \int_\R \int_0^\infty \bar F(k,v)
e^{ikz} J_0(rv) \cos \bigl(t \sqrt{k^2+ v^2}\bigr) v dv dk\,,
\end{equation}
with $\bar F (k,v): = \hankel_r \set{ \fourier_z \set{ F_\sigma
}}(k, v)$.

Substituting $\omega = \sqrt{k^2+v^2}$ in (\ref{sol2d}) and putting
$r=r_{\rm det}$ afterwards, leads to
\begin{equation}\label{press:data}
G_{\sigma}(z,t) = \frac{1}{2\pi} \int_\R \left(
\int_{\abs{k}}^\infty J_0 \bigl( r_{\rm det} \sqrt{\omega^2-k^2}
\bigr)    \omega  \bar F \bigl(k,\sqrt{\omega^2 - k^2}
\bigr)\cos(\omega t)  d\omega\right)  e^{ik z}  dk
\end{equation}
The inversion formulas for the Fourier and Cosine transforms now
imply that
\begin{equation*}
\cosine_t \set{ \fourier_z \set{ G_\sigma }}(k,\omega) =
\frac{\pi}{2} \left\{
  \begin{array}{ll}
    J_0\bigl( r_{\rm det} \sqrt{\omega^2-k^2} \bigr) \omega \bar F \bigl( k, \sqrt{\omega^2 - k^2}\bigr), &
\text{ if }  \omega > k \,,\\
    0, & \text{ otherwise} \,.
  \end{array}
\right.
\end{equation*}
Solving the last equation for $\bar F$ shows (\ref{invfo}).
\end{proof}

\noindent
Proposition \ref{thm:inv} implies that $F_\sigma$ can be reconstructed
from data $G_\sigma$ as follows:

\begin{enumerate}
\item[(i)]
The data $G_\sigma$ are Fourier and cosine transformed,
yielding to $\cosine_t \set{ \fourier_z \set{ G_\sigma }}$.

\item[(ii)]
According to (\ref{invfo}), $\cosine_t \set{ \fourier_z \set{ G_\sigma }}$
is mapped to $\hankel_r \set{ \fourier_z \set{ F_\sigma }}$.

\item[(iii)]
Finally, application of the inverse  Fourier and Hankel transforms yields
\[
    F_\sigma(z, r)
    =
   \frac{1}{2\pi} \int_\R
   \int_0^\infty
   \hankel_r \set{ \fourier_z \set{ F_\sigma }}(k, v)
   J_0(r v) e^{i k z} \ v dv dk \,.
   \]
\end{enumerate}

\begin{remark}[Instability of (\ref{invfo})]
Inversion formula (\ref{invfo}) is not defined when $J_0(r_{\rm det} v)$
equals $0$. From the proof of the above theorem it is clear that for
exact data
\begin{equation}
\cosine_t \{ \fourier_z \{ G_\sigma \}\}\bigl(k,\sqrt{v_n^2 + k^2}\bigr) = 0\,,
\qquad n \in \N \,,
\end{equation}
with $(v_n)_{n\in \N}$ denoting the zeros  of $r \mapsto J_0( r_{\rm det} v)$.
In practice, however, only noisy (approximately measured)
data $G_\sigma^{\delta}\simeq G_\sigma$ are available.
In general,
\[
    \cosine_t \{ \fourier_z \{G_\sigma^\delta \}\}
\bigl(k,\sqrt{v_n^2 + k^2}\bigr) \neq 0 \,.
\]
It is therefore difficult to stably evaluate the quotient in
(\ref{invfo}) in practice.
\end{remark}

\begin{psfrags}
\psfrag{R}{\footnotesize$R$}
\psfrag{r}{\footnotesize$r_0$}
\psfrag{R-r}{\footnotesize$R-r_{\rm det}$}
\begin{figure}[tbh!]
\begin{center}
\includegraphics[width=0.45\textwidth]{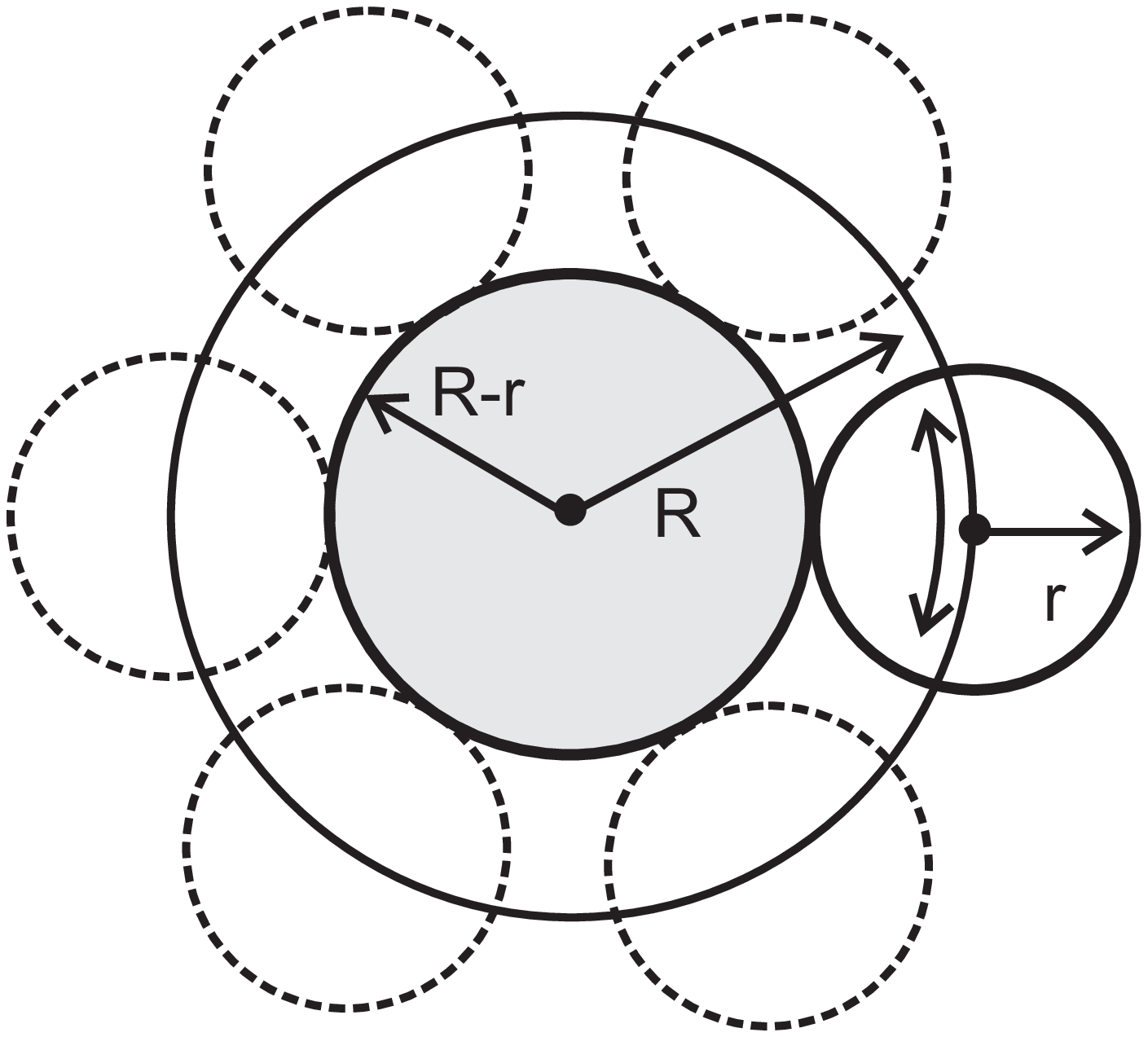}\quad
\includegraphics[width=0.45\textwidth]{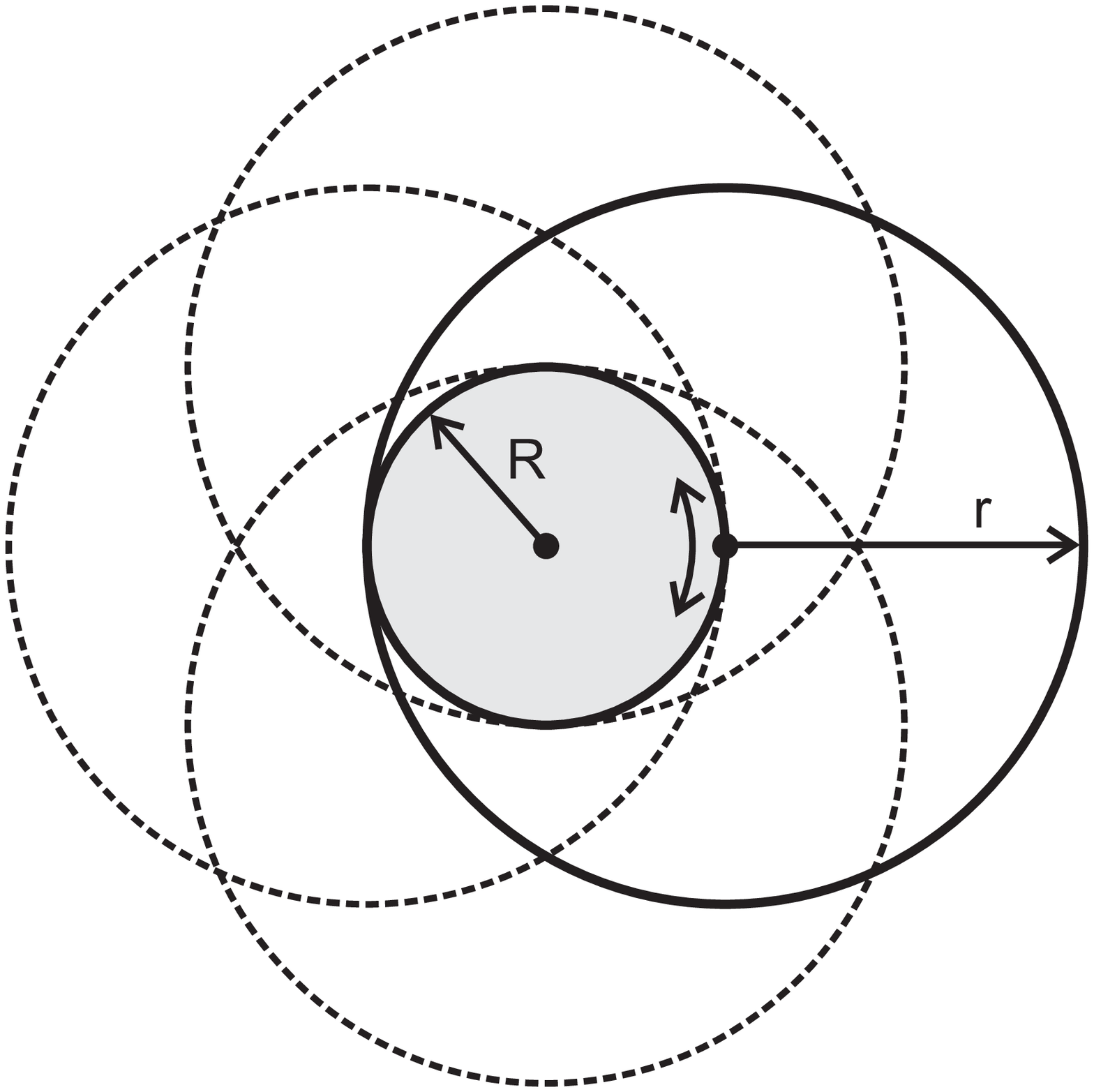}
\caption{Cross section of experimental buildup when $r_{\rm det} < \Rrc$
(left) and $r_{\rm det} \geq 2R$ (right). In both case object has to be
supported in the gray disc.}\label{fig:big+small}
\end{center}
\end{figure}
\end{psfrags}

Because the acoustic pressure is measured outside of the investigated object,
only the following two situations occur in practical applications,
see Figure \ref{fig:big+small}:
\begin{itemize}
  \item[(i)]
  The stack of circles is strictly outside the object.
  In this case $r_{\rm det} < \Rrc$ and  $\supp(f) \subset B_{\Rrc-r_{\rm det}}(0) \times \R$.

  \item[(ii)]
  The object is enclosed in the stack of circles. In this
  case $r_{\rm det} \geq 2 \Rrc$.
\end{itemize}
For the case $r_{\rm det} \geq 2 \Rrc$ we will provide two stable
formulas based on expansions in bases of special functions. The case
$r_{\rm det} < \Rrc$ turns out to be harder, we currently do not
have a stable alterative to (\ref{invfo}). In this case the function
$\fourier \set{F_\sigma}(k, \cdot)$ is not supported in the interval
$(0, r_{\rm det})$ and thus it cannot be expanded into a Fourier
Bessel series which is crucial in the proves of theorems
\ref{ser:expbR} and \ref{thm:inv2}. However, in the limiting case
$r_{\rm det} \ll \Rrc$ a stable reconstruction formula is obtained,
compare with Remark \ref{r0is0} below.

\section{Stable Inversion Formulas}
\label{sec:rec-stable}

In the following we
fix $f \in C_0^\infty(B_R(0)\times \R)$, $\sigma \in S^1$, define $F_\sigma$,
$G_\sigma$  by (\ref{eq:Fsigma}), (\ref{eq:data}), and let
$(v_n)_{n\in \N}$ denote the zeros of the function $v \mapsto J_0(
r_{\rm det} v)$.

\medskip
Our first stable inversion formula is as follows:
\begin{theorem}[The D'Hospital trick]
\label{ser:expbR} Assume $r_{\rm det} \geq 2 \Rrc$.
Then
\begin{equation}
F_\sigma(z,r) =
 \frac{2}{\pi^2 r_{\rm det}^3}
 \int_\R \left( \sum_{n  \in \mathbb{N} }
\frac{ \sine_t\set{t \fourier_z\set{G_{\sigma} }} \bigl(k ,\sqrt{k^2+ v_n^2
}\bigr)v_n }{k^2 + v_n^2} \frac{J_0(r v_n)}{J_1(r_{\rm det} v_n)^3}
\right)
e^{ikz} dk
\label{eq:stabf} 
\end{equation}
for any $(z,r) \in \R \times (0, \infty)$.
\end{theorem}

\begin{proof} The assumptions $f \in C_0^\infty(B_R(0)\times \R)$ and
$r_{\rm det} \geq 2 \Rrc$ imply  that $ \fourier_z\{ F_{\sigma}
\}(k, \cdot)$ is compactly supported in $(0,r_{\rm det})$. It can
therefore be expanded  in a Fourier Bessel series \cite{Sne72}
\begin{equation}\label{eq:fsera}
\fourier_z \set{F_\sigma}(k,r) = \frac{2}{r_{\rm det}^2} \sum_{n \in \mathbb{N}}
\hankel_r\set{\fourier_z\set{F_{\sigma} }}
(k, v_n)  \frac{J_0(r v_n)}{J_1( r_{\rm det} v_n )^2} \,.
\end{equation}
According to (\ref{invfo}) we have
\begin{equation*}
\hankel_r\set{\fourier_z\set{F_{\sigma} }}(k,v)   =\frac{2}{\pi}
\frac{
\cosine_t\set{\fourier_z\set{G_{\sigma} }}\bigl(k ,\sqrt{k^2+ v^2}\bigr)}{J_0(r_{\rm det}v) \sqrt{k^2+
v^2}} \,, \qquad v \not\in \set{v_n: n\in\N} \,.
\end{equation*}
Applying the rule of D'Hospital gives
\begin{align}\nonumber
\hankel_r\set{\fourier_z\set{F_{\sigma} }}(k,v_n)
&=
\frac{2}{\pi}\lim_{v \rightarrow v_n}
\frac{
\partial/\partial v \left[\cosine_t\set{\fourier_z\set{G_{\sigma} }}\bigl(k ,\sqrt{k^2+ v^2}\bigr)\right]}
{\partial/\partial v  \left[J_0(r_{\rm det}v) \sqrt{k^2+ v^2}\right]}
\\\nonumber
&= \frac{2}{\pi}\lim_{v \rightarrow v_n} \frac{\sine_t\set{t \fourier_z\set{G_{\sigma}
}} \bigl(k ,\sqrt{k^2+ v^2}\bigr) \frac{v}{\sqrt{k^2 + v^2}}}{ J_1( r_{\rm det} v)
r_{\rm det} \sqrt{k^2 + v^2 }}\\
&= \frac{2}{\pi r_{\rm det}}\frac{\sine_t\set{t \fourier_z\set{G_{\sigma}
}} \bigl(k ,\sqrt{k^2+ v_n^2}\bigr) v_n}{ J_1( r_{\rm det} v_n)
 \bigl( k^2 + v_n^2 \bigr)} \,. \label{hospa}
\end{align}
Inserting (\ref{hospa}) in (\ref{eq:fsera}) and using the Fourier
inversion formula shows (\ref{eq:stabf}).
\end{proof}

\begin{remark}\label{asymex}[Stabilty of \req{stabf}]
From the asymptotic approximation (see \cite{AbrSte72})
\begin{eqnarray*}
J_m (x) \simeq \sqrt{ \frac{2}{\pi x} } \cos \left(x- \frac{m
\pi}{2} - \frac{ \pi}{4} \right) \,, \quad \text{for } x \to
\infty\,,
\end{eqnarray*}
of the $m$-th order Bessel function it follows
that
\[
v_n \simeq  \frac{\pi (n + 1/4)}{r_{\rm det}}\,, \quad
| J_1(r_{\rm det} v_n)|
\simeq
\sqrt{\frac{2}{ \pi r_{\rm det} v_n   } }
\,, \quad
\text{ for }n \to \infty \,.
\]
Moreover the summands  in \req{stabf} take  the asymptotic form
\begin{align*}
    & \abs{\sine_t\set{t \fourier_z\set{G_{\sigma} }} \bigl(k ,\sqrt{k^2+ v_n^2}\bigr)
    \frac{v_n }{k^2 + v_n^2} \frac{  J_0(r v_n) }{ J_1(r_{\rm det} v_n)^3} }
    \\
    & \qquad
    \simeq
    \abs{\sine_t\set{t \fourier_z\set{G_{\sigma} }} \bigl(k ,\sqrt{k^2+ v_n^2}\bigr)}
    \frac{1 }{v_n}
    \frac{ \abs{ \cos \left(v_n - \frac{ \pi}{4} \right)} ( 2   / (\pi r v_n) )^{1/2}   }{ (2  / ( \pi r_{\rm det} v_n))^{3/2} }
    \\
    & \qquad
    \leq
    \frac{r^{1/2}}{4 r_{\rm det}^{3/2} }
    \abs{\sine_t\set{t \fourier_z\set{G_{\sigma} }} \bigl(k ,\sqrt{k^2+ v_n^2}\bigr)} \,.
\end{align*}
Consequently,  the parts  that do not depend on  the data $G_\sigma$
are bounded, and \req{stabf} can be implemented in stable way.
\end{remark}

\medskip
In the sequel we derive an additional inversion
formula that circumvents the division by zero problem. In fact, our
formula will be a consequence of the following result
derived in \cite{XuXuWan02}.

\begin{proposition}
Let $p$ denote the unique solution of
(\ref{eq:wave3d})-(\ref{eq:ini3db}) and let $f_\sigma^m$ and
$g_\sigma^m$ denote the Fourier coefficients of $f( \f\Phi_\sigma(
z, r, \alpha) )$ and
\[
    g_\sigma(\alpha, z, t)
    :=
    \left\{
      \begin{array}{ll}
        p ( \f\Phi_\sigma(z, r, \alpha) , t ), & \text{ for } t > 0 \,,\\
        0 , & \text{ otherwise} \,,
      \end{array}
    \right.
\]
with respect to $\alpha$. Then
\begin{equation}\label{eq:xuxuwang}
    \hankel_r \set{\fourier_t \set{ f_\sigma^m  }}(k, v) =
    \frac{2}{\pi} \frac{\fourier_t \set{\fourier_z \set{g_\sigma^m}} ( k, \sqrt{k^2 + v^2 } )  }{H^{2}_n(
    r_{\rm det} v)  \sqrt{k^2+ v^2}}
\,, \qquad (k, v) \in \R \times (0, \infty) \,,
\end{equation}
with $H^{2}_m$ denoting the $m$-th order second kind Hankel function.
\end{proposition}

\bigskip
Now the second stable inversion formula can be stated as follows:

\begin{theorem} \label{thm:inv2}
Assume $r_{\rm det} \geq 2 \Rrc$.
Then
\begin{equation}
F_\sigma(z,r) =
 \frac{2}{\pi^2 r_{\rm det}^2}
 \int_\R \left( \sum_{n  \in \mathbb{N} }
\frac{ \fourier_t \set{\fourier_z \set{G_\sigma}}
    ( k, \sqrt{k^2 + v_n^2 } )  }{H^{2}_0(
    r_{\rm det} v_n)  \sqrt{k^2+ v_n^2}}\frac{J_0(r v_n)}{J_1(r_{\rm det} v_n)^2}
\right)
e^{ikz} dk \,.
\label{eq:stabf2} 
\end{equation}
Here $G_\sigma$ is extended by $G_\sigma(z,t) = 0$ for $t<0$.
\end{theorem}

\begin{proof}
We use again the  Fourier-Bessel series (\ref{eq:fsera})  of proof of Theorem \ref{ser:expbR}.
Recalling the definitions of $F_\sigma$, $G_\sigma$
and the Fourier coefficients $f_\sigma^m$, $g_\sigma^m$
one notices that $F_\sigma = f_\sigma^0$, $G_\sigma = g_\sigma^0$.
Therefore (\ref{eq:xuxuwang}) for $m=0$ implies
\begin{multline}\label{eq:xuxuwang2}
    \hankel_r \set{\fourier_z \set{ F_\sigma  }}(k, v)
    =
     \frac{2}{\pi}  \frac{ \fourier_t \set{\fourier_z \set{G_\sigma}}
    ( k, \sqrt{k^2 + v^2 } )  }{H^{2}_0(
    r_{\rm det} v)  \sqrt{k^2+ v^2}} \,, \qquad (k, v) \in \R \times (0, \infty)
    \,.
\end{multline}
Inserting (\ref{eq:xuxuwang2}) in (\ref{eq:fsera}) and using the
Fourier inversion formula shows (\ref{eq:stabf2}).
\end{proof}

Equation (\ref{eq:xuxuwang2}) is quite similar to (\ref{invfo}).
However, in the denominator in (\ref{eq:xuxuwang2}) the zero order second kind
Hankel function appears (instead of the the zero order Bessel function) which cannot be zero
for a finite argument \cite{AbrSte72}.
Moreover, the asymptotic  expansion of the Bessel and the  second kind Hankel function
show that the summands in \req{stabf2} that do not depend on the data $G_\sigma$
remain bounded as $n \to \infty$.

\begin{remark}
The derivation of (\ref{eq:xuxuwang}) is based on the following
Green function expansion in cylindrical coordinates  \cite{XuXuWan02}
\begin{multline}\label{eq:green}
\frac{e^{-i \omega |\f\Phi_\sigma(z,r,\alpha)-\f\Phi_\sigma(z_0,r_{\rm det},\alpha_0) |}}
{|\f\Phi_\sigma(z,r,\alpha) - \f\Phi_\sigma(z_0,r_{\rm det},\alpha_0) |} \\
\\=
\frac{-i\pi}{2} \int _\R \left( \sum_{m \in\Z}
A_m( v r, v r_{\rm det}) e^{-im ( \alpha-\alpha_0)} \right)
e^{-i\omega ( z-z_0)} \, dz\,,
\end{multline}
with $v  = \sign(\omega) \sqrt{\abs{\omega^2- k^2}}$,
\begin{equation*}
A_m(vr, vr_{\rm det}) =
\left\{
  \begin{array}{ll}
    H_m^{(2)}(v r_{\rm det})  J_m( v r ) , & \text{ if } \omega^2 > k^2 \,, \\
    2i/\pi \, K_m( \abs{v} r_{\rm det} ) I_m( \abs{v} r )  , & \hbox{otherwise} \,,
  \end{array}
\right.
\end{equation*}
and $I_m$, $K_m$ denoting the $m$-th order modified Bessel functions
of first and second kind, respectively. Here its is assumed that $r_{\rm det} > r$.

Interchanging the roles of $r$ and $r_{\rm det}$ implies that for $r_{\rm det} > r$ the Green function
expansion \req{green} holds with
\begin{equation*}
A_m(vr, vr_{\rm det}) =
\left\{
  \begin{array}{ll}
    J_m( v r_{\rm det} ) H_m^{(2)}(v r)   , & \text{ if } \omega^2 > k^2 \,, \\
    2i/\pi \, I_m( \abs{v} r_{\rm det} ) K_m( \abs{v} r )   , & \hbox{otherwise} \,.
  \end{array}
\right.
\end{equation*}
Similar to the proof of \req{xuxuwang} in \cite{XuXuWan02} this leads to a  formula for reconstructing
$F_\sigma$ in the case $r_{\rm det} \leq 2R$, however again an unstable one with $J_0( v r_{\rm det} )$ in the denominator.
\end{remark}

\psfrag{K=4}{$K=4$}
\psfrag{K=100}{$K=100$}
\begin{figure}[tbh!]
\begin{center}
\includegraphics[width=0.45\textwidth]{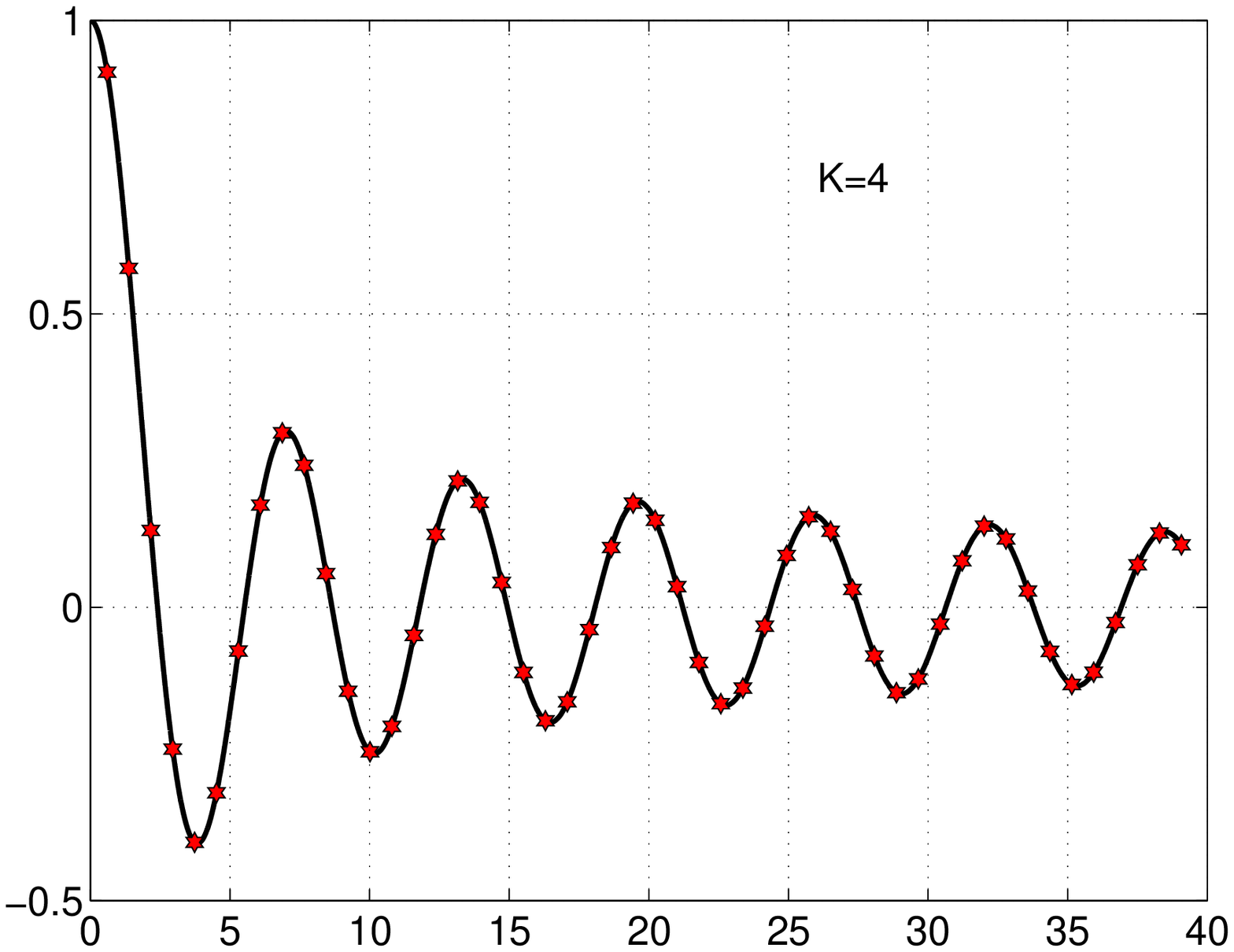} \qquad
\includegraphics[width=0.45\textwidth]{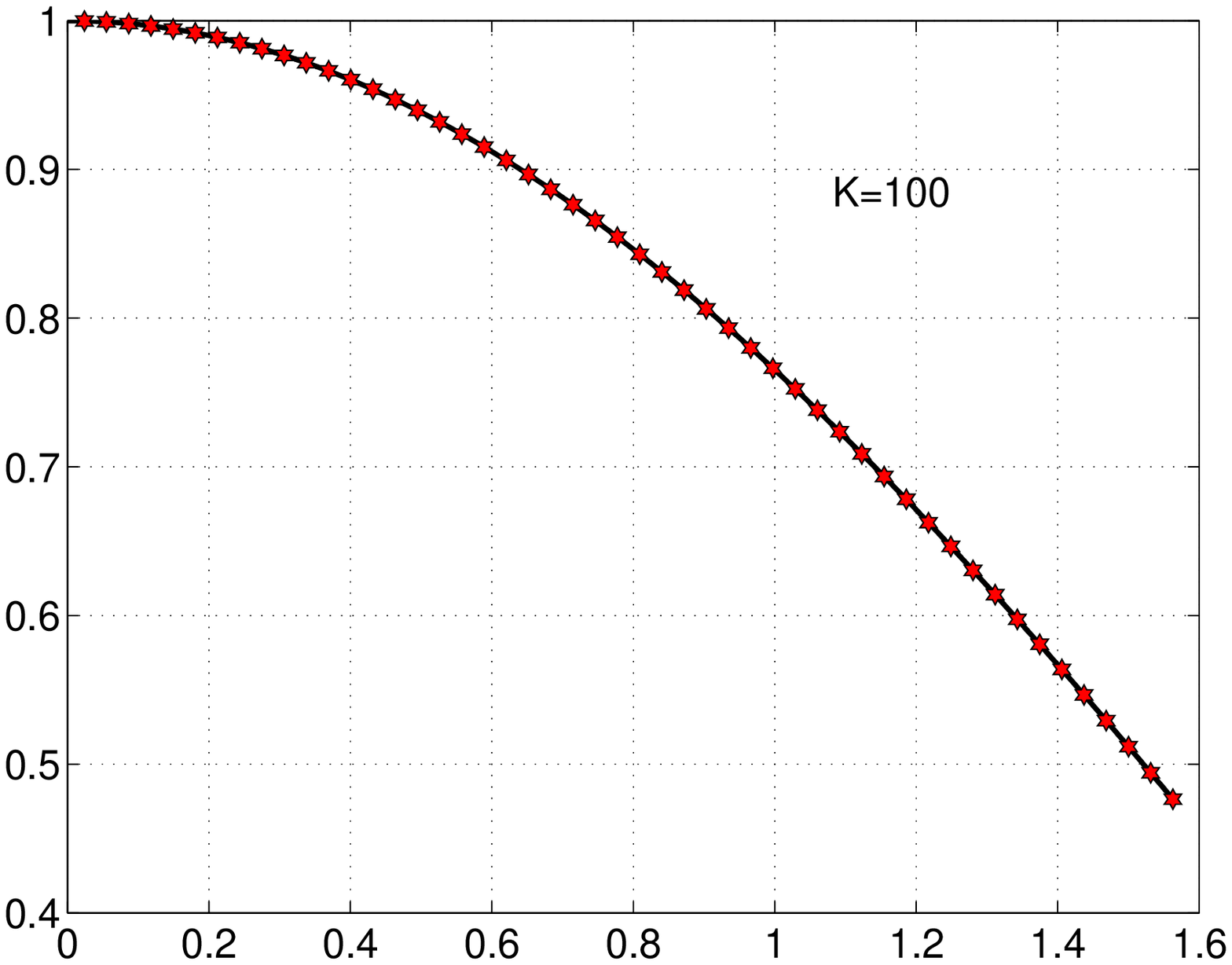}
\caption{The first $50$ denominators $J_0(r_{\rm det} v_n / K)$
in \req{fsera-2} for $K=4$ (left) and $K = 100$ (right).}\label{fig:smallR}
\end{center}
\end{figure}

\begin{remark}\label{r0is0}
Suppose that $r_{\rm det} < \Rrc$ and that $f$ is supported in
$B_{\Rrc-r_{\rm det}}(0) \times \R$.
For $r_{1} \geq 2\Rrc - r_{\rm det}$ let $(\tilde v_n)_{n \in \N}$ denote the zeros  of
$v \mapsto J_0( r_{1} \tilde v)$.
Then one can expand $\fourier_z \set { F_\sigma }(k, \cdot)$ in a Fourier Bessel series (see \cite{Sne72})
\begin{equation} \label{eq:sertilde}
\fourier_z \set{F_\sigma}(k,r)
=
\frac{2}{r_{1}^2} \sum_{n \in \mathbb{N}}
\hankel_r\set{\fourier_z\set{F_{\sigma} }}
(k, \tilde v_n)  \frac{J_0(r \tilde v_n)}{J_1( r_{1} \tilde v_n )^2}
\,, \qquad (k,r) \in \R \times (0, \infty) \,.
\end{equation}
According to (\ref{invfo}) we have
\begin{equation}\label{rel_data}
\hankel_r \set{ \fourier_z \set{ F_\sigma }}(k, v)
=
\frac{2}{\pi}
\frac{\cosine_t \set{ \fourier_z \set{ G_{\sigma} }} ( k,
\sqrt{k^2 + v^2 } )  }{J_0(r_{\rm det} v) \sqrt{k^2 + v^2 } }
\,\qquad v \not \in \set{v_n: n\in \N} \,.
\end{equation}
If we assume that $r_{1}$  is a integer multiple of
$r_{\rm det}$, i.e., $r_{1} = K r_{\rm det}$, then $\tilde v_n = v_n/K \not\in \set{v_m: m\in \N}$
for any $n \in \N$. Therefore, inserting (\ref{rel_data}) in \req{sertilde} yields
\begin{equation}\label{eq:fsera-2}
\fourier_z \set{F_\sigma}(k,r)
=
\frac{2}{\pi^2 r_{1}^2} \sum_{n \in \mathbb{N}}
\frac{\cosine_t \set{ \fourier_z \set{ G_{\sigma} }} ( k,
\sqrt{k^2 + \tilde v_n^2 } )  }{J_0(r_{\rm det} \tilde v_n) \sqrt{k^2 + \tilde v_n^2 } }
\frac{J_0(r \tilde v_n)}{J_1( r_{1} \tilde v_n )^2}
\end{equation}
In general, \req{fsera-2} is a again sensible to noise when
$\tilde v_n$ gets close to a zero of $J_0(r_{\rm det} v)$.

In the limiting case $r_{\rm det} \ll \Rrc$ and for $n$ not to large, the denominators $J_0(r_{1}\tilde v_n)$ are well bounded from below
(see right image in Figure \ref{fig:smallR}).  In this case, truncating \req{fsera-2} leads to a stable inversion formula.
In particular, for $r_{\rm det} = 0$ one obtains
\begin{equation}\label{eq:fsera-3}
\fourier_z \set{F_\sigma}(k,r)
=
\frac{2}{\pi^2 r_{1}^2} \sum_{n \in \mathbb{N}}
\frac{\cosine_t \set{ \fourier_z \set{ G_{\sigma} }} ( k,
\sqrt{k^2 + \tilde v_n^2 } ) }{\sqrt{k^2 + \tilde v_n^2 } }
\frac{J_0(r \tilde v_n)}{J_1( r_{1} \tilde v_n )^2}
\end{equation}
for any $r_{1} \geq 2\Rrc$. Together with \req{inv-finchi} this
provides a novel inversion formula for PAT using point-like
detectors on a cylindrical recording surface.
\end{remark}

\section{Numerical Experiments}

In practice one deals with discrete measurement data
\[
\Ge_{\lo} [\mo,\no]
:= G_{\sigma_{\lo}}( z_{\mo} ,t_{\no})
\,,
\qquad
(\lo, \mo, \no) \in \{1,\dots, \No_\sigma \}\times \{1,\dots, \No_z \} \times \{1,\dots, \No_t \} \,,
\]
where $G_{\sigma}$ is as in \req{data}, and where $\sigma_{\lo} = 2
\pi (\lo -1)/\No_{\sigma}$, $z_{\mo} =  H ( \mo-1)  / \No_z$ and
$t_{\no} = T (\no -1)/ \No_t$ are discrete samples of the angle,
height and time, respectively. Here $H > 0$ represents the finite
height of the stack of circular integrating detectors (see Figure
\ref{fig:1}) and  $T$ is such that $G_{\sigma}( z ,t) = 0$ for
$t\geq T$ and $z \in[0, H]$.

In this section we outline how to implement \req{stabf} and
\req{stabf2}  in order to find an approximation
\[
\Fe_{\lo} [\mo,\no]
\simeq
F_{\sigma_{\lo}}( z_{\mo} , r_{\no})
 \,,
\qquad
(\mo, \no) \in \{1,\dots, \No_z \} \times \{1,\dots, \No_r \} \,,
\]
with $r_{\no} = r_{\rm det} (\no -1)/ \No_r$.
Having calculated such an approximation,
one can reconstruct a discrete approximation to $f$ by applying the filtered back-projection algorithm
of \cite{FinHalRak07} for fixed $\mo$, see Remark \ref{rem:twostep}.

\smallskip
A numerical reconstruction method based on \req{stabf} is as follows:
\begin{itemize}
\item[(i)]
The discrete Fourier transform  (with respect to the first component)
of the data
\begin{equation} \label{eq:FFT}
\Feo \set{\Ge_\lo} [ \mo ,\no ]:= \sum_{\mo' =1}^{\No_z} \Ge_\lo[ \mo',\no]~
e^{-i 2 \pi \mo  (\mo'-1) / \No_z }
\end{equation}
with $(\mo, \no)  \in \{- \No_z/2 ,\dots , \No_z/2-1\} \times \{1 ,\dots , \No_t\}$,
is considered as an approximation to $\fourier \set{F_{\sigma_\lo}}(2\pi \mo/H, t_\no)$.
\item[(ii)]
The sine transform $\sine\set{t \fourier \set{F_{\sigma_\lo}}}$, evaluated at
\[
    \omega_{\mo, \no} = \sqrt{ (2 \pi \mo/H)^2 + v_{\no}^2 }\,,
    (\mo, \no) \in  \{- \No_z/2 ,\dots , \No_z/2-1\} \times\{0 ,\dots , \No_r-1  \} \,,
\]
is approximated by the trapezoidal rule, leading to
\begin{equation}\label{Sin_trafo}
\Seo \set{ \f t \Feo \set {\Ge_\lo }}[ \mo , \no  ]:= \sum_{\no' =1 }^{\No_t}
t_{\no'} \Feo \set { \Ge_\lo } [\mo, \no'] \sin \bigl(\omega_{\mo, \no} t_{\no'}  \bigr)\,.
\end{equation}
\item[(iii)] Finally, truncating the Fourier Bessel Series and approximating
the inverse Fourier transform with the trapezoidal rule leads to
discrete version of (\ref{eq:stabf}):
\begin{equation}\label{invfo_dis}
\Feo_\lo [\mo, \no]
:=
\frac{4 T}{\pi  r_{\rm det}^3 \No_t}
\sum_{\mo' = -\No_z}^{\No_z/2-1}
\sum_{\no'=0}^{ \No_r-1}  \frac{ v_{\no'} \Seo  \set{\f t \Feo
\set{ \Ge_\lo}}[\mo', \no']}{\omega_{\mo', \no'}^2 J_1(r_{\rm det} v_{\no'})^3 }
J_0(r_\no v_{\no'})
e^{-i 2 \pi \mo'  (\mo-1) / \No_z }\,,
\end{equation}
with $(\mo, \no) \in  \{- \No_z/2 ,\dots , \No_z/2-1\} \times\{0
,\dots , \No_r-1  \}$ in formula (\ref{invfo_dis})
\end{itemize}

\smallskip
A numerical reconstruction method using  \req{stabf} is be  obtained in
an analogous manner. In this case one replaces  (\ref{invfo_dis})  by
\begin{equation}\label{invfo_dis2}
\Feo_\lo [\mo, \no] := \frac{4 T}{\pi  r_{\rm det}^2 \No_t}
\sum_{\mo' = -\No_z}^{\No_z/2-1} \sum_{\no'=0}^{ \No_r-1}  \frac{
 \Feo  \set{ \Feo \set{ \Ge_\lo}}[\mo',
\no']}{\omega_{\mo', \no'} H_0^1(r_{\rm det} v_{\no'}) J_0(r_{\rm
det} v_{\no'})^2 } J_0(r_\no v_{\no'}) e^{-i 2 \pi \mo'  (\mo-1) /
\No_z }\,,
\end{equation}
which is the discrete analogue   of \req{stabf2}.

To give a rough estimate of the computational complexity for the
previous calculations let us assume $\No_z = \No_r = \No_t = \No_\sigma =: \No$ and that
the values of the sine function and the Bessel function are pre-computed and
stored in lookup tables. Then the evaluation \req{FFT}
needs $\mathcal O(\No^2 \log \No)$ floating point operations (FLOPS) whereas (\ref{Sin_trafo}) and
(\ref{invfo_dis}) require  $\mathcal O(\No^3)$ FLOPS.
The filtered back projection formula \req{inv-finchi} also requires
$\mathcal O(\No^3)$ FLOPS, see \cite{FinHalRak07}.
For three dimensional reconstruction \req{FFT}, (\ref{Sin_trafo}), (\ref{invfo_dis})
and the filtered back-projection formula have to be applied $\No$ times. Hence
the total number of FLOPS is estimated as
\begin{eqnarray}
\No_{\rm FLOPS} = \No \left( \mathcal O(\No^2 \log \No) + \mathcal O(\No^3) +  \mathcal O(\No^3)\right)  =  \mathcal O(\No^4) \,.
\end{eqnarray}
Note that three dimensional back-projection type formulas which use point measurement data
have complexity $\mathcal O(\No^5)$.

\begin{figure}[tbh!]
\begin{center}
\includegraphics[height=0.9\textwidth, width=0.3\textwidth]{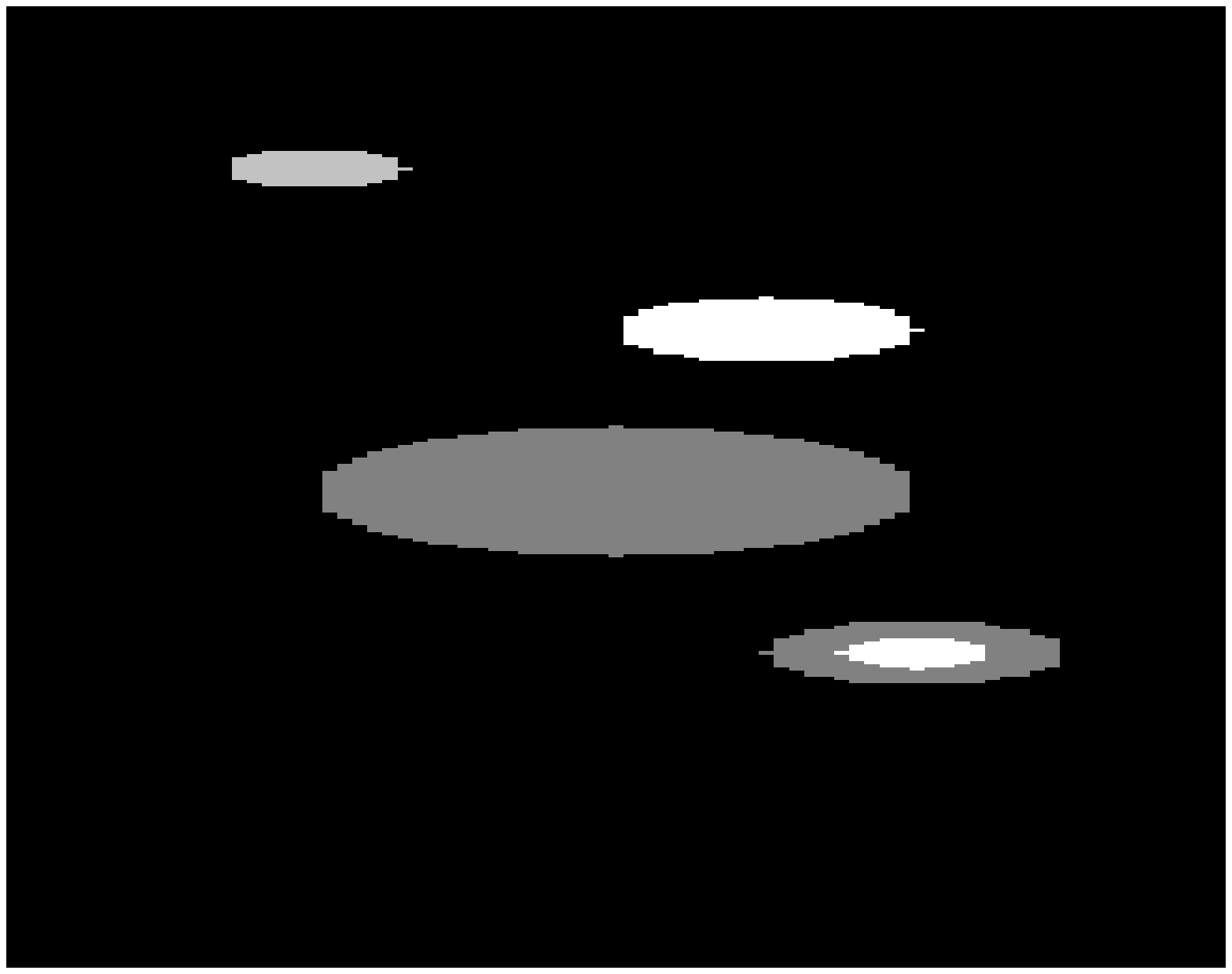}\hspace{0.1\textwidth}
\includegraphics[height=0.9\textwidth,width=0.3\textwidth]{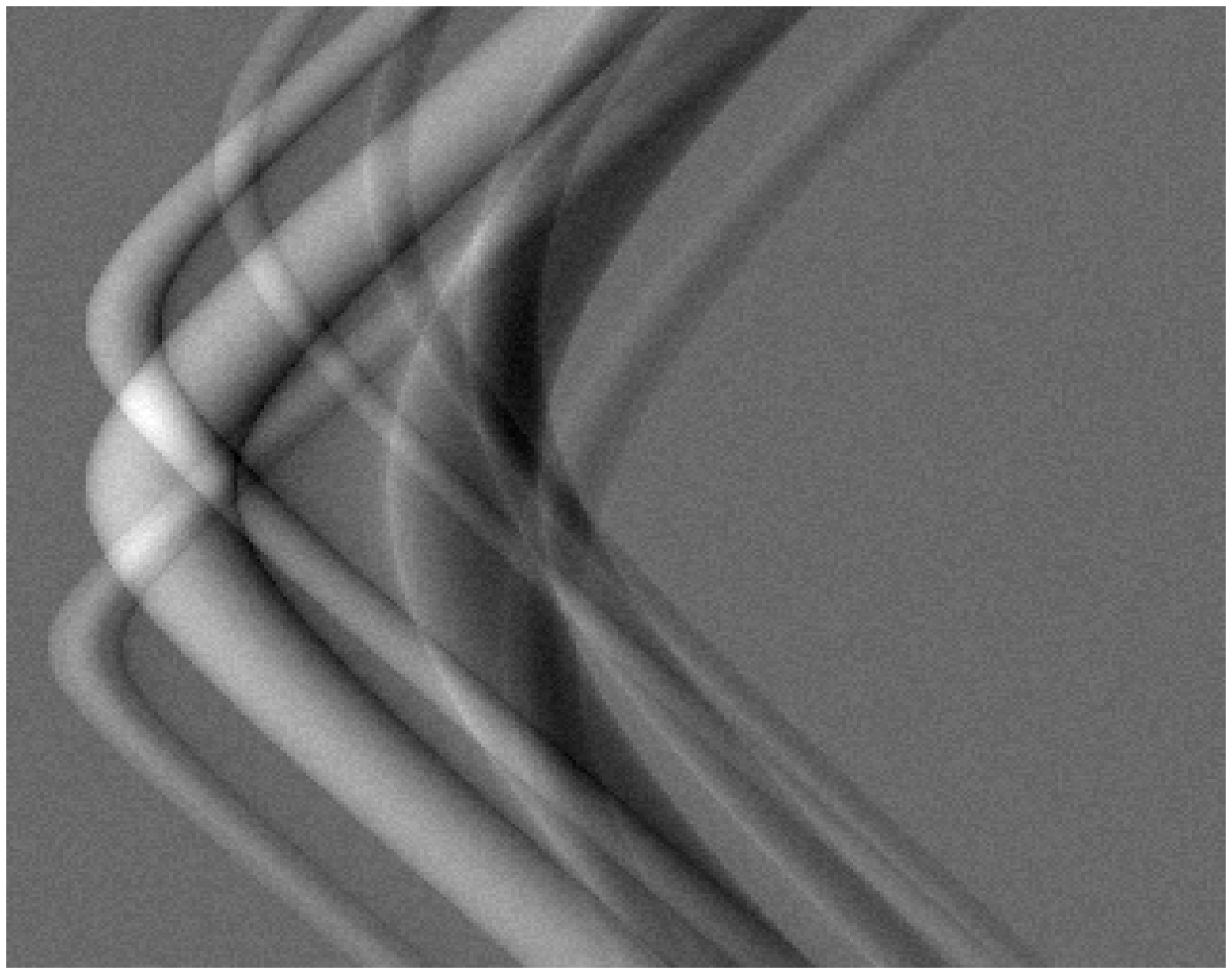}
\caption{Left: Cross section of five absorbing spheres ($z$ versus $r$). Right: The
measurement data with $10 \%$ Gaussian noise added ($z$ versus $t$).}\label{fig-phant}
\end{center}
\end{figure}

\begin{figure}[tbh!]
\begin{center}
\includegraphics[height =  0.9\textwidth,width=0.3\textwidth]{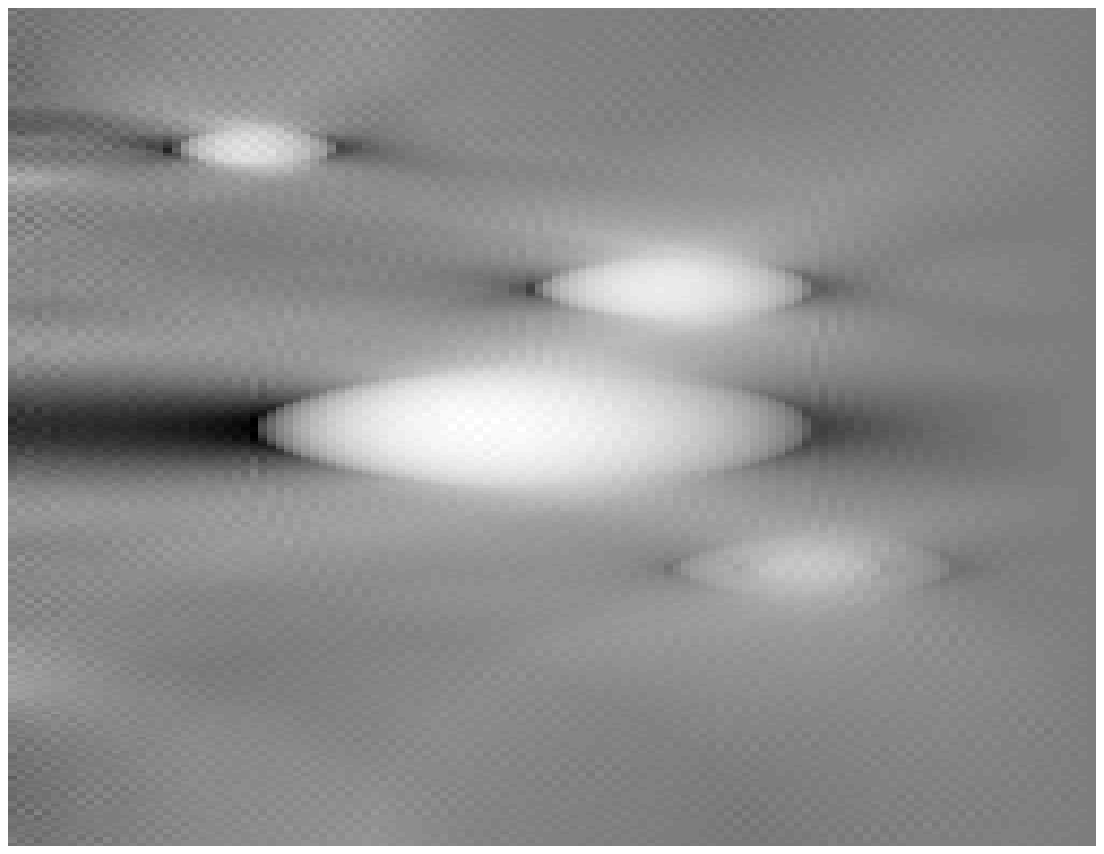}\hspace{0.1\textwidth}
\includegraphics[height =  0.9 \textwidth,width=0.3\textwidth]{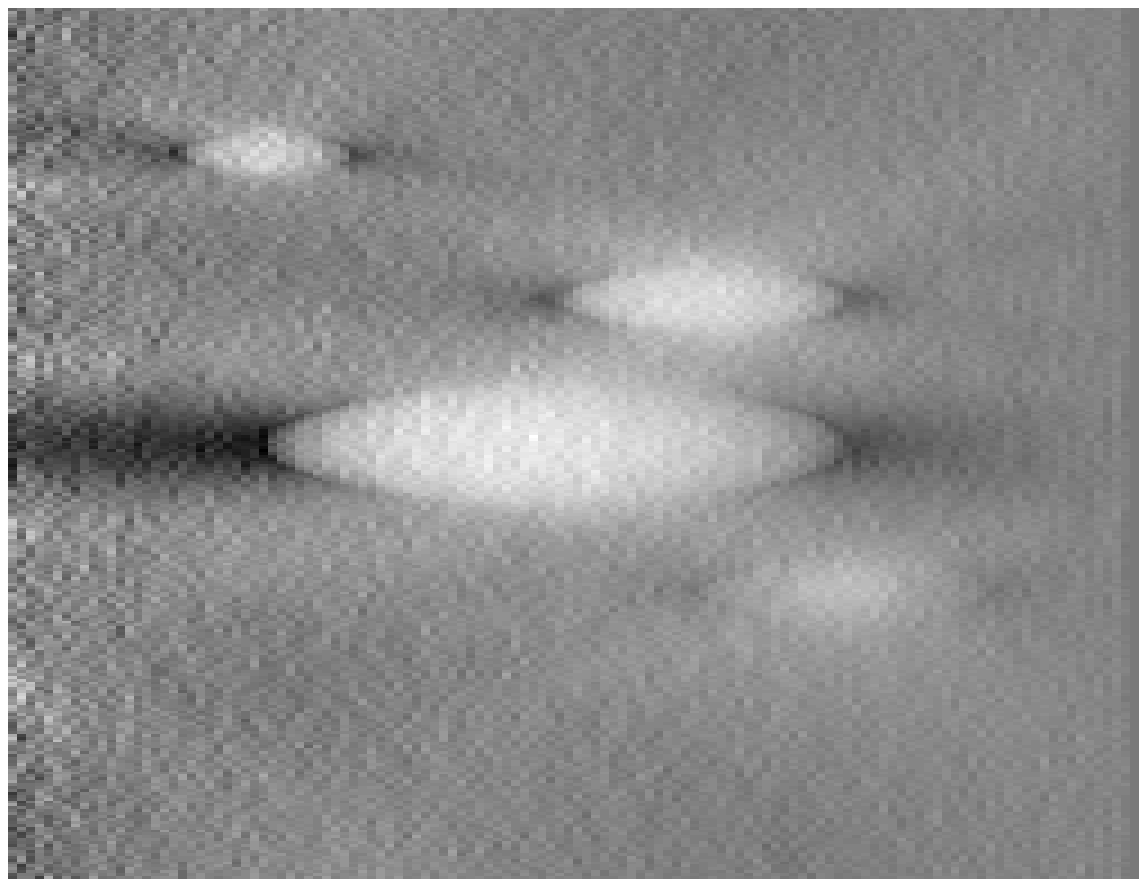}
\caption{Reconstruction with \req{stabf} from simulated (left) and noisy data (right). }\label{fig-rec}
\end{center}
\end{figure}

\begin{figure}[tbh!]
\begin{center}
\includegraphics[height =  0.9\textwidth,width=0.3\textwidth]{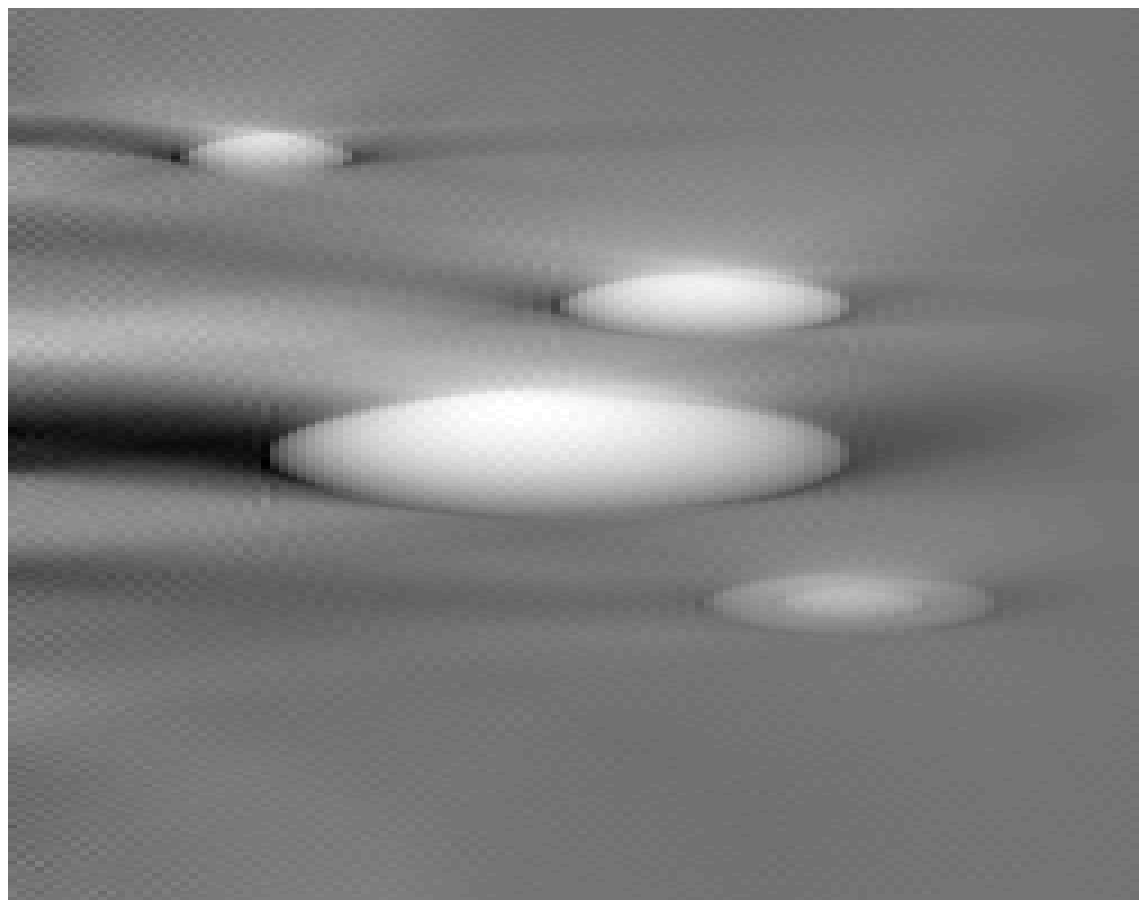}\hspace{0.1\textwidth}
\includegraphics[height =  0.9\textwidth,width=0.3\textwidth]{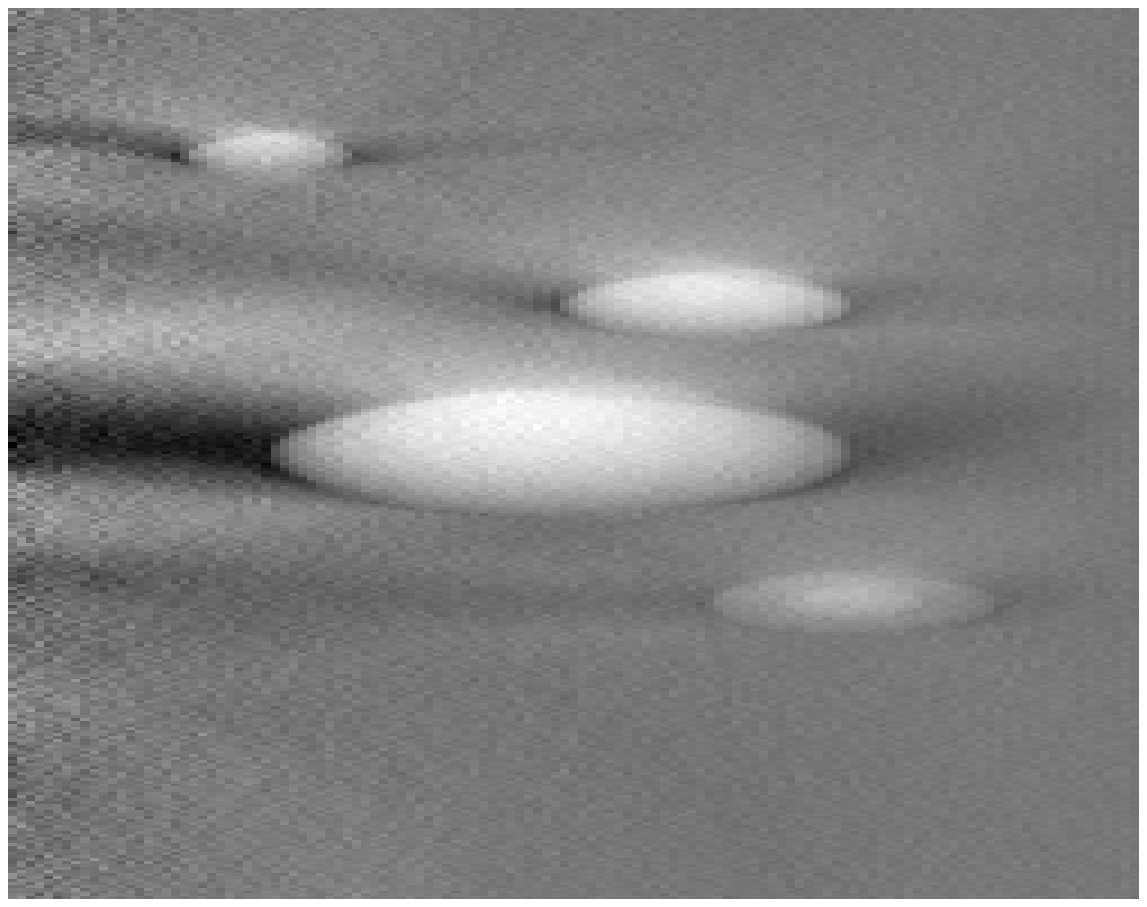}
\caption{Reconstruction with \req{stabf2} from simulated (left) and noisy data
(right). }\label{fig-rec2}
\end{center}
\end{figure}

%

In the following numerical experiments we take $\Rrc = 0.4$,
$H=3.75$ and $T=4$. The synthetic initial data $f$ is assumed to be
a superposition of radially symmetric objects around centers $\f
x_n$, i.e.,
\[
    f(\f x)
    =
    \sum_n f_n \bigl(\| \f x- \f x_n\|\bigr)
    \,, \qquad \f x \in \R^3\,.
\]
The acoustic pressure generated by a single radially symmetric
object at position $\f x$ and time $t$ is given by (see
\cite{HalSchuSch05})
\begin{eqnarray}\label{pres:spher}
    p_n(\f x,t)
    =
    \frac{ \norm{\f x -\f x_n }- t }{2\norm{\f x -\f x_n}}
    \, f_n
    \Bigl( \bigl| \norm{\f x - \f x_n} - t \bigr| \Bigr).
\end{eqnarray}
By the superposition principle the total pressure is
\begin{equation*}
p(\f x ,t )=\sum_{n=1}^N p_n(\f x, t )
\,, \qquad  (\f x, t) \in \R^3 \times (0, \infty)\,.
\end{equation*}
The measurement data $G_\sigma(z,t) = 1/(2\pi) \int_0^{2\pi} p(\Phi_\sigma( z, r_{\rm det}, \alpha),t ) d\alpha$,
see \req{Psigma}, (\ref{eq:data}), were generated by evaluating of (\ref{pres:spher}) followed by numerical
integration over $\alpha$. The radius $r_{\rm \det}$ of the circular integrating detectors is chosen to be
$2\Rrc$. In this case the stack of circular integrating detectors fully encloses the synthetic
initial data $f$, see right image in in Figure  \ref{fig:big+small}.

Figure \ref{fig-phant} shows a vertical cross section of the initial
pressure $f$ and the measurement data $G_\sigma$ where Gaussian
noise with a variance of $10\%$ of the maximal data valued is added.
The stack of circular integrating detectors is centered to the left
of the objects. The presented discrete implementation $\No_t=320$
measurements in time and $\No_z = 300$ in space. In both
reconstructions the value $\No_r$ was chosen to be $\No_r = 130$

The reconstructions of $F_\sigma$ with (\ref{invfo_dis}) from exact
and noisy data are depicted in Figure \ref{fig-rec} from formula
(\ref{invfo_dis}) and with formula (\ref{invfo_dis2}) in Figure
\ref{fig-rec2}. In the reconstructed images one notices some
blurred boundaries which are limited data artifacts \cite{LouQui00,
Qui93, XuYWanAmbKuc04} arising from the finite height of the stack
of circular integrating detectors. Moreover the images reconstructed
with (\ref{invfo_dis2}) are less sensitive to noise.

\section{Conclusion}
In this article a novel experimental buildup for PAT using circular
integrating detectors was proposed. For collecting measurement data
a fiber-based Mach-Zehnder or Fabry-Perot interferometer can be used
as an circular integrating detector. We showed that the 3D imaging
problem reduces to a series of 2D problems. This decomposition can
be used to reduce the operation count of derived reconstruction
algorithms. We derived two stable exact reconstruction formulas,  \req{stabf} and \req{stabf2},
for the case that the object is contained in the stack of detecting
circles.  In the case where the object is outside the detecting circles, a stable reconstruction
formula is obtained for the limiting case $r_{\rm det} \ll \Rrc$. As a byproduct, this leads to
a novel reconstruction formula \req{fsera-3} for PAT using point detectors on a cylindrical
surface.

\section*{Acknowledgement}
This work has been supported by the Austrian Science Foundation (FWF)
within the framework of the NFN ``Photoacoustic Imaging in Biology and Medicine'', Project S10505-N20. Moreover, the work of M. Haltmeier has been supported by the technology transfer office
of the University Innsbruck (transIT).

\def\cprime{$'$} \def\cprime{$'$} \def\cprime{$'$}


\begin{thebibliography}{10}

\bibitem{AbrSte72}
M.~Abramowitz and I.A. Stegun.
\newblock {\em Handbook of Mathematical Functions}.
\newblock Dover, New York, 1972.

\bibitem{AgrKucQui07}
M.~L. Agranovsky, K.~Kuchment, and E.~T. Quinto.
\newblock Range descriptions for the spherical mean {R}adon transform.
\newblock {\em J. Funct. Anal.}, 248(2):344--386, 2007.

\bibitem{BurHofPalHalSch05}
P.~Burgholzer, C.~Hofer, G.~Paltauf, M.~Haltmeier, and O.~Scherzer.
\newblock Thermo\-acoustic tomography with integrating area and line detectors.
\newblock {\em IEEE Trans. Ultrason., Ferroeletr., Freq. Control},
  52(9):1577--1583, 2005.

\bibitem{BurMatHalPal07}
P.~Burgholzer, G.~J. Matt, M.~Haltmeier, and G.~Paltauf.
\newblock Exact and approximate imaging methods for photoacoustic tomography
  using an arbitrary detection surface.
\newblock {\em Phys. Rev. E}, 75(4):046706, 2007.

\bibitem{ClaKli07}
C.~Clason and M.~Klibanov.
\newblock The quasi-reversibility method for thermoacoustic tomography in a
  heterogeneous medium.
\newblock {\em SIAM J. Sci. Comp.}, 2007.
\newblock accepted.

\bibitem{Dep07}
C.~Depeursinge, editor.
\newblock {\em Novel Optical Instrumentation for Biomedical Applications III},
  volume 6631 of {\em Proceedings of SPIE-OSA}, 2007.

\bibitem{EseLarLarDeyMotPro02}
R.~O. Esenaliev, I.~V. Larina, K.~V. Larin, D.~J. Deyo, M.~Motamedi, and D.~S.
  Prough.
\newblock Optoacoustic technique for noninvasive monitoring of blood
  oxygenation: a feasibility study.
\newblock {\em App. Opt.}, 41(22):4722--4731, 2002.

\bibitem{FinHalRak07}
D.~Finch, M.~Haltmeier, and Rakesh.
\newblock Inversion of spherical means and the wave equation in even
  dimensions.
\newblock {\em SIAM J. Appl. Math.}, 68(2):392--412, 2007.

\bibitem{FinRak07}
D.~Finch and Rakesh.
\newblock The spherical mean value operator with centers on a sphere.
\newblock {\em Inverse Probl.}, 23(6):37--49, 2007.

\bibitem{GruHalPalBur07}
H.~Gr{\"u}n, M.~Haltmeier, G.~Paltauf, and P.~Burgholzer.
\newblock Photoacoustic tomography using a fiber based {F}abry--{P}erot
  interferometer as an integrating line detector and image reconstruction by
  model-based time reversal method.
\newblock In {\em \cite{Dep07}}, 2007.

\bibitem{HalSchBurNusPal07}
M.~Haltmeier, O.~Scherzer, P.~Burgholzer, R.~Nuster, and G.~Paltauf.
\newblock Thermo\-acoustic tomography \& the circular {R}adon transform: Exact
  inversion formula.
\newblock {\em Math. Models Methods Appl. Sci.}, 17(4):635--655, 2007.

\bibitem{HalSchBurPal04}
M.~Haltmeier, O.~Scherzer, P.~Burgholzer, and G.~Paltauf.
\newblock Thermoacoustic imaging with large planar receivers.
\newblock {\em Inverse Probl.}, 20(5):1663--1673, 2004.

\bibitem{HalSchuSch05}
M.~Haltmeier, T.~Schuster, and O.~Scherzer.
\newblock Filtered backprojection for thermoacoustic computed tomography in
  spherical geometry.
\newblock {\em Math. Methods Appl. Sci.}, 28(16):1919--1937, 2005.

\bibitem{HriKucNgu08}
Y.~Hristova, P.~Kuchment, and L.~Nguyen.
\newblock Reconstruction and time reversal in thermoacoustic tomography in
  acoustically homogeneous and inhomogeneous media.
\newblock {\em Inverse Problems}, 24(5):055006 (25pp), 2008.

\bibitem{KolHonSteMul03}
R.~G.~M. Kolkman, E.~Hondebrink, W.~Steenbergen, and F.~F.~M. De~Mul.
\newblock In vivo photoacoustic imaging of blood vessels using an
  extreme-narrow aperture sensor.
\newblock {\em IEEE J. Sel. Topics Quantum Electron.}, 9(2):343--346, 2003.

\bibitem{KruMilReyKisReiKru00}
R.~A. Kruger, K.~D. Miller, H.~E. Reynolds, W.~L. Kiser, D.~R. Reinecke, and
  G.~A. Kruger.
\newblock Breast cancer in vivo: contrast enhancement with thermoacoustic {CT}
  at 434 {MH}z-feasibility study.
\newblock {\em Radiology}, 216(1):279--283, 2000.

\bibitem{KucKun08}
P.~Kuchment and L.~A. Kunyansky.
\newblock Mathematics of thermoacoustic and photoacoustic tomography.
\newblock {\em European J. Appl. Math.}, 19:191--224, 2008.

\bibitem{Kun07}
L.~A. Kunyansky.
\newblock Explicit inversion formulae for the spherical mean {R}adon transform.
\newblock {\em Inverse Probl.}, 23(1):373--383, 2007.

\bibitem{Kun07b}
L.~A. Kunyansky.
\newblock A series solution and a fast algorithm for the inversion of the
  spherical mean radon transform.
\newblock {\em Inverse Probl.}, 23(6):S11--S20, 2007.

\bibitem{LouQui00}
A.K. Louis and E.T. Quinto.
\newblock Local tomographic methods in sonar.
\newblock In {\em Surveys on solution methods for inverse problems}, pages
  147--154. Springer, Vienna, 2000.

\bibitem{ManKhaHesSteLee05}
S.~Manohar, A.~Kharine, J.~C.~G. van Hespen, W.~Steenbergen, and T.~G. van
  Leeuwen.
\newblock The twente photoacoustic mammoscope: system overview and performance.
\newblock {\em Physics in Medicine and Biology}, 50(11):2543--2557, 2005.

\bibitem{PalNusHalBur07}
G.~Paltauf, R.~Nuster, M.~Haltmeier, and P.~Burgholzer.
\newblock Experimental evaluation of reconstruction algorithms for limited view
  photoacoustic tomography with line detectors.
\newblock {\em Inverse Probl.}, 23(6):81--94, 2007.

\bibitem{PatSch07}
S.~K. Patch and O.~Scherzer.
\newblock Special section on photo- and thermoacoustic imaging.
\newblock {\em Inverse Probl.}, 23:S1--S122, 2007.

\bibitem{Qui93}
E.~T. Quinto.
\newblock Singularities of the {X}-ray transform and limited data tomography in
  {${\bf R}\sp 2$} and {${\bf R}\sp 3$}.
\newblock {\em SIAM Journal on Mathematical Analysis}, 24(5):1215--1225, 1993.

\bibitem{SchGraGroHalLen08}
O.~Scherzer, M.~Grasmair, H.~Grossauer, M.~Haltmeier, and F.~Lenzen.
\newblock {\em Variational Methods in Imaging}, volume 167 of {\em Applied
  Mathematical Sciences}.
\newblock Springer, New York, 2008.

\bibitem{Sne72}
I.~N. Sneddon.
\newblock {\em The Use of Integral Transforms}.
\newblock McGraw-Hill, New York, 1972.

\bibitem{XuWan03}
M.~Xu and L.~V. Wang.
\newblock Analytic explanation of spatial resolution related to bandwidth and
  detector aperture size in thermoacoustic or photoacoustic reconstruction.
\newblock {\em Phys. Rev. E}, 67(5):0566051--05660515 (electronic), 2003.

\bibitem{XuWan06}
M.~Xu and L.~V. Wang.
\newblock Photoacoustic imaging in biomedicine.
\newblock {\em Rev. Sci. Instruments}, 77(4):041101, 2006.

\bibitem{XuYWanAmbKuc04}
Y.~Xu, L.~V. Wang, G.~Ambartsoumian, and P.~Kuchment.
\newblock Reconstructions in limited-view thermoacoustic tomography.
\newblock {\em Med. Phys.}, 31(4):724--733, 2004.

\bibitem{XuXuWan02}
Y.~Xu, M.~Xu, and L.~V. Wang.
\newblock Exact frequency-domain reconstruction for thermoacoustic
  tomography--{II}: Cylindrical geometry.
\newblock {\em IEEE Trans. Med. Imag.}, 21:829--833, 2002.

\bibitem{YanWan07}
X.~Yang and L.~V. Wang.
\newblock Ring-based ultrasonic virtual point detector with applications to
  photoacoustic tomography.
\newblock {\em Applies Physics Letters}, 90(25):251103, 2007.

\bibitem{ZanSchHal09}
G.~Zangerl, M.~Haltmeier, and O.~Scherzer.
\newblock Circular integrating detectors in photo and thermoacoustic
  tomography.
\newblock {\em Inverse Probl. Sci. Eng.}, 17(1):133--142, 2009.

\end{thebibliography}
\end{document}